\documentclass[12pt,a4paper]{amsart}

\makeatletter
\tagsleft@false
\makeatother
\usepackage{mathrsfs}
\usepackage{tikz}
\usepackage{amsfonts}
\usepackage{amsmath}
\usepackage{amssymb}
\usepackage{stmaryrd}
\usepackage{enumerate}
\usepackage{hyperref}
\usepackage{latexsym}
\usepackage{xcolor}
\usepackage[shortlabels]{enumitem}
\usepackage[numbers,sort&compress]
{natbib}
\setlist[enumerate,1]{label={\upshape(\roman*)}}
\makeatletter

\newcommand{\Rmnum}[1]
{\expandafter\@slowromancap\romannumeral #1@}
\makeatother
\newtheorem{thm}{Theorem}[section]

\newtheorem{lemma}[thm]{Lemma}

\newtheorem{example}[thm]{Example}
\newtheorem{defin}[thm]{Definition}

\newtheorem{fact}[thm]{Fact}

\theoremstyle{definition}
\newtheorem{remark}[thm]{Remark}

\def\wz{\tilde}

\title[Weakly distance-regular digraphs of diameter 2]{Weakly distance-regular digraphs of diameter 2}

\date{}

\thanks{*Corresponding author}

\author[Wang]{Xiangli Wang}
\address{Laboratory of Mathematics and Complex Systems (MOE),~School of Mathematical Sciences\\Beijing Normal University\\Beijing 100875\\China}
\email{wangxl@mail.bnu.edu.cn}

\author[Yang]{Yuefeng Yang*}
\address{School of Science\\China University of Geosciences\\Beijing 100083\\China}
\email{yangyf@cugb.edu.cn}

\begin{document}
	
\begin{abstract}
Weakly distance-regular digraphs is a directed version of distance-regular graphs. In this paper, we characterize all weakly distance-regular digraphs of diameter $2$.
\end{abstract}
	
	
	
\maketitle
\section{Introduction}
As a natural generalization of distance-regular graphs (for the background of distance-regular graphs, refer to \cite{AEB98,DKT16}) and distance-regular digraphs (see the definition in Section 2), Wang and Suzuki \cite{KSW03} proposed the concept of weakly distance-regular digraphs. Distance-regular graphs, distance-regular digraphs and weakly distance-regular digraphs can all be regarded as association schemes.

A \emph{$d$-class association scheme} $\mathfrak{X}$ is a pair $(X,\{R_{i}\}_{i=0}^{d})$, where $X$ is a finite set, and each $R_{i}$ is a
nonempty subset of $X\times X$ satisfying the following axioms (see \cite{EB21,EB84,PHZ96,PHZ05} for a background of the theory of association schemes):
\begin{enumerate}
\item \label{1} $R_{0}=\{(x,x)\mid x\in X\}$ is the diagonal relation;
	
\item \label{2} $X\times X=R_{0}\cup R_{1}\cup\cdots\cup R_{d}$, $R_{i}\cap R_{j}=\emptyset~(i\neq j)$;
	
\item \label{3} for each $i$, $R_{i}^{\top}=R_{i^{*}}$ for some $0\leq i^{*}\leq d$, where $R_{i}^{\top}=\{(y,x)\mid(x,y)\in R_{i}\}$;
	
\item \label{4} for all $i,j,h$, the cardinality of the set $$P_{i,j}(x,y):=R_{i}(x)\cap R_{j^{*}}(y)$$ is constant whenever $(x,y)\in R_{h}$, where $R(x)=\{y\mid (x,y)\in R\}$ for $R\subseteq X\times X$ and $x\in X$. This constant is denoted by $p_{i,j}^{h}$.
\end{enumerate}
A $d$-class association scheme is also called an association scheme with $d$ classes. The integers $p_{i,j}^{h}$ are called the \emph{intersection numbers} of $\mathfrak{X}$. We say that $\mathfrak{X}$ is \emph{commutative} if $p_{i,j}^{h}=p_{j,i}^{h}$ for all $0\leq i,j,h\leq d$. The subsets $R_{i}$ are called the \emph{relations} of $\mathfrak{X}$. For each $i$, the integer $k_{i}$ $(=p_{i,i^{*}}^{0})$ is called the \emph{valency} of $R_{i}$. A relation $R_{i}$ is called \emph{symmetric} if $i=i^{*}$, and \emph{non-symmetric} otherwise. An association scheme is called \emph{symmetric} if all relations are symmetric, and \emph{non-symmetric} otherwise.

In order to state our main result, we introduce some basic notations about weakly distance-regular digraphs. A \emph{digraph} $\Gamma$ is a pair $(V\Gamma, A\Gamma)$ where $V\Gamma$ is a finite nonempty set of vertices and $A\Gamma$ is a set of ordered pairs (\emph{arcs}) $(x, y)$ with distinct vertices $x$ and $y$. For any arc $(x,y)\in A\Gamma$, if $A\Gamma$ also contains an arc $(y,x)$, then $\{(x,y),(y,x)\}$ can be viewed as an {\em edge}. We say that $\Gamma$ is an \emph{undirected graph} or a {\em graph} if $A\Gamma$ is a symmetric relation. For the digraphs $\Gamma$ and $\Sigma$, the \emph{lexicographic product} of $\Gamma$ and $\Sigma$ is the digraph with the vertex set $V\Gamma\times V\Sigma$ such that $((u_1, u_2),(v_1, v_2))$ is an arc if and only if either $(u_1, v_1)\in A\Gamma$, or $u_1 = v_1$ and $(u_2, v_2)\in A\Sigma$.

A \emph{path of length} $r$ from $x$ to $y$ in the digraph $\Gamma$ is a finite sequence of vertices $(x=w_{0},w_{1},\ldots,w_{r}=y)$ such that $(w_{t-1}, w_{t})\in A\Gamma$ for $1\leq t\leq r$. A digraph (resp. graph) is said to be \emph{strongly connected} (resp. \emph{connected}) if, for any vertices $x$ and $y$, there is a path from $x$ to $y$. The length of a shortest path from $x$ to $y$ is called the \emph{distance} from $x$ to $y$ in $\Gamma$, denoted by $\partial_\Gamma(x,y)$. The maximum value of distance function in $\Gamma$ is called the \emph{diameter} of $\Gamma$. We define $\Gamma_{i}$ ($0\leq i\leq d$) to be the set of ordered pairs $(x,y)$ with $\partial_{\Gamma}(x,y)=i$, where $d$ is the diameter of $\Gamma$. A path $(w_{0},w_{1},\ldots,w_{r-1})$ is called a \emph{circuit of length} $r$ when $(w_{r-1},w_0)\in A\Gamma$. The \emph{girth} of $\Gamma$ is the length of a shortest circuit in $\Gamma$. Let $\wz{\partial}_{\Gamma}(x,y):=(\partial_{\Gamma}(x,y),\partial_{\Gamma}(y,x))$ be the \emph{two-way distance} from $x$ to $y$ in $\Gamma$, and $\wz{\partial}(\Gamma)$ the set of all pairs $\wz{\partial}_{\Gamma}(x,y)$. For any $\wz{i}\in\wz{\partial}(\Gamma)$, we define $\Gamma_{\wz{i}}$ to be the set of ordered pairs $(x,y)$ with $\wz{\partial}_{\Gamma}(x,y)=\wz{i}$.

A strongly connected digraph $\Gamma$ is said to be \emph{weakly distance-regular} if the configuration $\mathfrak{X}(\Gamma)=(V\Gamma,\{\Gamma_{\tilde{i}}\}_{\tilde{i}\in\tilde{\partial}(\Gamma)})$ is an association scheme. We call $\mathfrak{X}(\Gamma)$ the \emph{attached scheme} of $\Gamma$. We say that $\Gamma$ is \emph{commutative} if $\mathfrak{X}(\Gamma)$ is commutative. For more information about weakly distance-regular digraphs, see \cite{YF22,AM1,HS04,LS,YLS,KSW03,KSW04,YYF16,YYF18,YYF20,YYF22,YYF,QZ23,QZ24}.

Distance-regular graphs of diameter $2$ (also known as strongly regular graphs; see \cite{AEB98,AEB22,PJ99,CG01,JJ79} for details on the general theory of strongly regular graphs) are essentially the same as symmetric $2$-class association schemes. Distance-regular digraphs of diameter $2$ are equivalent to non-symmetric $2$-class association schemes. As a generazation of distance-regular graphs and distance-regular digraphs, this paper characterizes all weakly distance-regular digraphs of diameter $2$. See Section 2 for precise definitions of a P-polynomial association scheme, a subscheme, a wreath product and a wedge product of association schemes.

\begin{thm}\label{th1}
Let $\mathfrak{X}=(X,\{R_{l}\}_{l=0}^{d})$ be an association scheme. The digraph $(X,A)$ is a weakly distance-regular digraph of diameter $2$ with $\mathfrak{X}$ as its attached scheme if and only if $\mathfrak{X}$ has exactly one pair of non-symmetric relations, say $R_1^{\top}=R_2$, and one of the following holds:
\begin{enumerate}
\item\label{th1-1} $d=2$ and $A=R_i$;
		
\item\label{th1-3} $d=3$ and $A=R_i\cup R_3$;
		
\item\label{th1-2} $d=3$, $A=R_i$, and $\mathfrak{X}$ is neither a $P$-polynomial association scheme nor a wedge product of subschemes $(\langle R_1\rangle(x),\{R_h\cap([\langle R_1\rangle(x)]\times [\langle R_1\rangle(x)])\}_{h=0}^2)$ for each $x\in X$ and a $1$-class association scheme;
		
\item\label{th1-4} $d=4$, $A=R_i\cup R_j$, and $\mathfrak{X}$ is neither a wreath product of a $1$-class association scheme and a $P$-polynomial association scheme nor a wedge product of subschemes $(\langle R_1,R_j\rangle(x),\{R_h\cap([\langle R_1,R_j\rangle(x)]\times [\langle R_1,R_j\rangle(x)])\}_{h\in \{0,1,2,j\}})$ for each $x\in X$ and a $1$-class association scheme.
\end{enumerate}
Here, $i\in\{1,2\}$ and $j\in\{3,4\}$.
\end{thm}
The remaining of this paper is organized as follows. In Section 2, we provide the required concepts and notations about association schemes. In Section 3, we give a proof of Theorem \ref{th1}.

\section{Association schemes}
In this section, we always assume that $\mathfrak{X}:=(X,R)$ is an association scheme, where $R=\{R_{i}\}_{i=0}^{d}$. In the following, we present some basic concepts, notations and results concerning association schemes.

First we list basic properties of intersection numbers which are used frequently in the remainder of this paper.

\begin{lemma}\label{jb}
{\rm (\cite[Proposition 5.1]{ZA99} and \cite[Chapter \Rmnum{2}, Proposition 2.2]{EB84})} The following hold:
\begin{enumerate}
\item\label{jb-1} $k_{i}k_{j}=\sum^{d}_{h=0}p_{i,j}^{h}k_{h}$;
	
\item\label{jb-2}  $p^{h}_{i,j}k_{h}=p^{i}_{h,j^{*}}k_{i}=p^{j}_{i^{*},h}k_{j}$;
		
\item\label{jb-3}  $\sum^{d}_{j=0}p^{h}_{i,j}=k_{i}$;
		
\item\label{jb-4} $\sum^{d}_{r=0}p^{r}_{e,l}p^{h}_{m,r}=\sum^{d}_{t=0}p^{t}_{m,e}p^{h}_{t,l}$.
\end{enumerate}
\end{lemma}

The \emph{adjacency matrix} $A_{i}$ of $R_{i}$ is the $|X|\times |X|$ matrix whose $(x,y)$-entry is $1$ if $(x,y)\in R_{i}$, and $0$ otherwise. By the \emph{adjacency} or \emph{Bose-Mesner algebra} $\mathfrak{U}$ of $\mathfrak{X}$ we mean the algebra generated by $A_0, A_1, \ldots, A_d$ over the complex field. Axioms \ref{1}--\ref{4} are equivalent to the following:
\[
A_0 = I_{|X|}, \quad \sum_{i=0}^d A_i = J_{|X|}, \quad A_i^\top = A_{i^*}, \quad A_i A_j = \sum_{h=0}^d p_{i,j}^h A_h,
\]
where $I_{|X|}$ and $J_{|X|}$ are the identity and all-one matrices of order $|X|$, respectively.

We say that $\mathfrak{X}$ is a \emph{$P$-polynomial} association scheme with respect to the ordering $R_0,R_1,\ldots,R_d$, if there exist some complex coefficient polynomials $v_{i}(x)$ of degree $i$ ($0\leq i\leq d$) such that $A_i=v_{i}(A_1).$

Distance-regular graphs are equivalent to symmetric association schemes which are $P$-polynomial. As a directed version of distance-regular graph, Damerell \cite{RMD81} introduced the concept of distance-regular digraphs. A digraph $\Gamma$ of diameter $d$ is said to be \emph{distance-regular} if $(V\Gamma,\{\Gamma_{i}\}_{0\leq i\leq d})$ is a non-symmetric association scheme. Notice that, for a distance-regular digraph $\Gamma$, $(V\Gamma,\{\Gamma_{i}\}_{0\leq i\leq d})$ is a $P$-polynomial association scheme with respect to the ordering $\Gamma_0,\Gamma_1,\ldots,\Gamma_d$. On the contrary, every $P$-polynomial non-symmetric association scheme arises from a distance-regular digraph in this way.

In \cite{RMD81}, Damerell  presented the following basic results concerning distance-regular digraphs.

\begin{thm}\label{drdg}
{\rm (\cite[Theorems 2 and 4]{RMD81})} Let $\Gamma$ be a distance-regular digraph of diameter $d$ and girth $g$. Then $d=g-1$ (short type) or $d=g$ (long type). Moreover, if $d=g$, then $\Gamma$ is a lexicographic product of a short distance-regular digraph of diameter $g-1$ and an empty graph.
\end{thm}

Let $E,F$ be two nonempty subsets of $R$. Define
$$EF:=\{R_{h}\mid\underset{R_i\in E, R_j\in F}{\sum}p_{i,j}^{h}\neq0\}.$$
We write $R_iR_j$ instead of $\{R_i\}\{R_j\}$, and $R^{2}_i$ instead of $\{R_i\}\{R_i\}$. If $R_{i^*}R_j\subseteq F$ for any $R_i, R_j\in F$, we say that $F$ is \emph{closed}. For $x\in X$, set $E(x):=\{y\in X\mid (x,y)\in \bigcup_{e\in E}e\}$.

\begin{thm} \label{th2.3}
Let $x\in X$ and $F$ be a closed subset of $R$. Then $(F(x),\{f\cap([F(x)]\times [F(x)])\}_{f\in F})$ is an association scheme.
\end{thm}
\begin{proof}
	This follows from \cite[Theorem 1.5.1 (ii)]{PHZ96}.
\end{proof}
Let $x\in X$ and $F$ be a closed subset of $R$. Then $(F(x),\{f\cap([F(x)]\times [F(x)])\}_{f\in F})$ is called a \emph{subscheme} of $\mathfrak{X}$. Let $X/F=\{ F(x) \mid x\in X\}$, $R_i^F=\{(F(x),F(y))\mid y\in FR_iF(x)\}$  for $0\leq i\leq d$ and $R \sslash F=\{R_i^F\mid R_i\in R\}$. By \cite[Theorem 4.1.3 (i)]{PHZ05}, the pair $(X/F, R \sslash F)$ forms an association scheme, called the \emph{quotient scheme} of $\mathfrak{X}$ over $F$.

Now we recall the definition of a wreath product of two association schemes. This notion appears in the monograph \cite{We76} by Weisfeiler, in a more general context of coherent configurations. Let $\mathfrak{Y}=(Y,\{S_j\}_{j=0}^e)$ be an association scheme. Then the \emph{wreath product} of $\mathfrak{X}$ and $\mathfrak{Y}$ denoted by $\mathfrak{X}  \leftthreetimes \mathfrak{Y}$ is defined on the set $X\times Y$ with the relations $W_0, W_1,\ldots, W_{d+e}$ satisfying the following:
\begin{align}
W_0&=\{((x,y),(x,y))\mid (x,y)\in X\times Y\};\nonumber\\
W_k&=\{((x_1,y),(x_2,y))\mid (x_1,x_2)\in R_k, y\in Y\},~\mbox{for}~ 1\leq k\leq d; \nonumber\\
W_k&=\{((x_1,y_1),(x_2,y_2))\mid x_1,x_2\in X, (y_1,y_2)\in S_{k-d}\}, ~\mbox{for}~ d+1\leq k\leq d+e.\nonumber
\end{align}
Then, it is easy to see that $\mathfrak{X}  \leftthreetimes \mathfrak{Y}=(X\times Y,\{W_k\}_{k=0}^{d+e})$ is an association scheme.

As a generalization of the wreath product, the wedge  product (generalized wreath product) of schemes has been defined by Muzychuk \cite{M} as a generalization of the wedge product of Schur rings. By \cite[Theorem 2.3]{M}, we have the following equivalent definition.

Let $K\subseteq F$ be two closed subsets of $R$ such that:
\begin{itemize}
\item [{\rm(a)}] $\sum_{R_i\in K}A_iA_j=\sum_{R_i\in K}k_iA_j=\sum_{R_i\in K}A_jA_i$, for every $R_j\in R\backslash F$;

\item [{\rm(b)}] $KR_i=R_iK$ for all $0\leq i\leq d$.
\end{itemize}
Then $\mathfrak{X}$ is the \emph{wedge  product} of subschemes $(F(x),\{f\cap([F(x)]\times [F(x)])\}_{f\in F})$ for each $x\in X$ and the quotient scheme $(X/K, R \sslash K)$.

Let $F$ be a nonempty closed subset of $R$. For a digraph $\Gamma$ with $X$ as its vertex set, the \emph{quotient digraph} of $\Gamma$ over $F$, denoted by $\Gamma/F$, is defined as the digraph with the vertex set $X/F$ in which $( F(x),F(y))$ is an arc whenever there is an arc in $\Gamma$ from $F(x)$ to $F(y)$.

Let $K$ be a nonempty subset of $R$. We define $\langle K\rangle$ to be the smallest closed subset containing $K$. If $K=\{R_{i_1},R_{i_2},\ldots,R_{i_l}\}$ for $1\leq l\leq d+1$, we write $\langle R_{i_1},R_{i_2},\ldots,R_{i_l}\rangle$ instead of $\langle \{R_{i_1},R_{i_2},\ldots,R_{i_l}\} \rangle$. We close this section with a result concerning weakly distance-regular digraphs.

\begin{lemma} \label{lem}
Let $\mathfrak{X}$ be a commutative association scheme such that $R_1^{\top}=R_2$ and $\langle R_a\rangle=\{R_0,R_a\}$ with $3\leq a\leq d$. Suppose $p_{1,1^{*}}^{a}=k_1$. Assume $\Delta_l=(X,R_1\cup R_l)$ for $1\leq l\leq d$. The following hold:
\begin{enumerate}
\item\label{lem-1} $\Delta_1$ is the lexicographic product of $\Delta_1/\langle R_a\rangle$ and the empty graph $\overline{K}_{k_a+1}$;
		
\item\label{lem-2} $\Delta_a$ is the lexicographic product of $\Delta_a/\langle R_a\rangle$ and the complete graph $K_{k_a+1}$.
\end{enumerate}
Moreover, if $\Delta_l$ is a weakly distance-regular digraph with $\mathfrak{X}$ as its attached scheme for some $l\in\{1,a\}$, then $\Delta_l/\langle R_a\rangle$ is also weakly distance-regular.
\end{lemma}
\begin{proof}
	
Let $l\in\{1,a\}$ and $\Lambda_l=\Delta_l/\langle R_a\rangle$. We claim that $(\langle R_a\rangle(x),\langle R_a\rangle(y))$ is an arc in $\Lambda_l$ if and only if $(x',y')$ is an arc in $\Delta_l$ for all $x'\in \langle R_a\rangle(x)$ and $y'\in\langle R_a\rangle(y)$. The sufficiency is immediate. Now suppose that $(\langle R_a\rangle(x),\langle R_a\rangle(y))$ is an arc in $\Lambda_l$. Then $y\notin \langle R_a\rangle(x)$. Without loss of generality, we may assume that $(x,y)$ is an arc in $\Delta_l$. Since $A\Delta_l=R_1\cup R_l$ with $l\in\{1,a\}$, we have $(x,y)\in R_1$. Let $y'\in \langle R_a\rangle(y)$. If $y=y'$, then $(x,y')\in R_1$; if $y\neq y'$, then $(y,y')\in R_a$ since $\langle R_a\rangle =\{R_0,R_a\}$, which implies $x\in R_{1^*}(y)=P_{1^*,1}(y,y')$ by the commutativity of $\mathfrak{X}$. Then $(x,y')\in R_1$. Let $x'\in \langle R_a\rangle(x)$. If $x'=x$, then $(x',y')\in R_1$; if $x'\neq x$, then $(x,x')\in R_a$ since $\langle R_a\rangle =\{R_0,R_a\}$, which implies $y'\in R_{1}(x)=P_{1,1^*}(x,x')$, and so $(x',y')\in R_1$. Since $x'\in \langle R_a\rangle(x)$ and $y'\in\langle R_a\rangle(y)$ were arbitrary, the necessity is valid. Thus, our claim is valid.
	
Since $\langle R_a\rangle =\{R_0,R_a\}$ and $A\Delta_l=R_1\cup R_l$, from the claim, $\Delta_1$ is the lexicographic product of $\Lambda_1$ and the empty graph $\overline{K}_{k_a+1}$, and $\Delta_a$ is the lexicographic product of $\Lambda_a$ and the complete graph $K_{k_a+1}$. Thus, \ref{lem-1} and \ref{lem-2} are both valid.
	
Now suppose that $\Delta_l$ is a weakly distance-regular digraph with $\mathfrak{X}$ as its attached scheme for some $l\in\{1,a\}$. Let $x_0,y\in V\Delta_l$ and $y\notin \langle R_a\rangle(x_0)$. Pick a shortest path $(x_0, x_1,\ldots,x_{m-1},y)$ in $\Delta_l$ with $m>0$. Since $\langle R_a\rangle =\{R_0,R_a\}$ and $A\Delta_l=R_1\cup R_l$, we have $\partial_{\Lambda_l}(\langle R_a\rangle(x_0),\langle R_a\rangle(y))\leq \partial_{\Delta_l}(x_0,y)$.
	
Pick a shortest path $(\langle R_a\rangle(x_0), \langle R_a\rangle(x_1),\ldots ,\langle R_a\rangle(x_{n-1}),\langle R_a\rangle(y))$ in $\Lambda_l$ with $n>0$. By the claim, $(x_0, x_1, \ldots,x_{n-1},y)$ is a path in $\Delta_l$. It follows that $\partial_{\Delta_l}(x_0,y)\leq \partial_{\Lambda_l}(\langle R_a\rangle(x_0),\langle R_a\rangle(y))$. Thus, $\partial_{\Delta_l}(x_0,y)=\partial_{\Lambda_l}(\langle R_a\rangle(x_0),\langle R_a\rangle(y))$.
	
Let $\tilde{i},\tilde{j},\tilde{h}\in\tilde{\partial}(\Lambda_l)\setminus\{(0,0)\}$ and $(\langle R_a\rangle(u),\langle R_a\rangle(v))\in(\Lambda_l)_{\tilde{h}}$. Since $\partial_{\Delta_l}(x,y)=\partial_{\Lambda_l}(\langle R_a\rangle(x),\langle R_a\rangle(y))$ for all $x,y\in X$ and $y\notin \langle R_a\rangle(x)$, we have $\tilde{i},\tilde{j}\in\tilde{\partial}(\Delta_l)$ and $(u,v)\in(\Delta_l)_{\tilde{h}}$. Since $\mathfrak{X}$ is the attached scheme of $\Delta_l$, we may assume $(\Delta_{l})_{\tilde{i}}=R_i$, $(\Delta_{l})_{\tilde{j}}=R_j$ and $(u,v)\in R_h$ for some  $i,j,h\in\{1,2,\ldots,a-1,a+1,\ldots,d\}$. Since $\partial_{\Delta_l}(x,y)=\partial_{\Lambda_l}(\langle R_a\rangle(x),\langle R_a\rangle(y))$ for all $x,y\in X$ and $y\notin \langle R_a\rangle(x)$, one gets
\begin{align}
R_i(u)\cap R_{j^*}(v)&=(\Delta_l)_{{\tilde{i}}}(u)\cap (\Delta_l)_{\tilde{j}^*}(v)\nonumber\\
    &=\cup_{w\in (\Delta_l)_{{\tilde{i}}}(u)\cap (\Delta_l)_{\tilde{j}^*}(v)}\langle R_a\rangle(w)\nonumber\\
	&=\cup_{\langle R_a\rangle(w)\in(\Lambda_l)_{\tilde{i}}(\langle R_a\rangle(u))\cap(\Lambda_l)_{\tilde{j}^*}(\langle R_a\rangle(v))}\langle R_a\rangle(w).\nonumber
\end{align}
Since $|\langle R_a\rangle(w)|=k_{a}+1$ for all $w\in X$, we have $$|(\Lambda_l)_{\tilde{i}}(\langle R_a\rangle(u))\cap(\Lambda_l)_{\tilde{j}^*}(\langle R_a\rangle(v))|=\frac{p_{i,j}^{h}}{k_{a}+1}.$$

Thus,  $\Lambda_l$ is also a weakly distance-regular digraph.
\end{proof}

\section{Proof of Theorem \ref{th1}}
In this section, we always assume that $\Gamma:=(X,A\Gamma)$ is a weakly distance-regular digraph and $\mathfrak{X}:=(X,\{R_i\}_{i=0}^d)$ is an association scheme with $R_1^{\top}=R_2$ and $R_{h}^{\top}=R_{h}$ for $3\leq h\leq d$. Before giving a proof of Theorem \ref{th1}, we need some auxiliary facts and lemmas.

\begin{fact}\label{sub}
If $\mathfrak{X}$ is the attached scheme of $\Gamma$, then $R_{i}\subseteq A\Gamma$ or $R_{i}\cap A\Gamma=\emptyset$ for all $1\leq i\leq d$.
\end{fact}

\begin{fact}\label{fac2}
The digraph $(X,(A\Gamma)^{\top})$ is weakly distance regular, and has the same attached scheme with $\Gamma$.
\end{fact}

In the remainder of this paper, we may assume that $\Gamma$ has diameter $2$. Since
\begin{align}\label{two way distance set}
\{(0,0),(1,2),(2,1)\}\subseteq\tilde{\partial}(\Gamma)\subseteq \{(0,0),(1,2),(2,1),(1,1),(2,2)\},
\end{align}
its attached scheme has at most $4$ classes. By a theorem of Higman \cite{HC75}, association schemes with at most four classes are commutative.

\subsection{The attached scheme of $\Gamma$ has $2$ classes} In this subsection, we consider the case that the attached scheme of $\Gamma$ has $2$ classes.

\begin{lemma}\label{lem3.3}
If $d=2$, then $(X,R_i)$ is a weakly distance-regular digraph of diameter $2$ with $\mathfrak{X}$ as the attached scheme for $i\in\{1,2\}$.
\end{lemma}
\begin{proof}
Let $\Delta=(X,R_1)$. Since $A\Delta=R_1$ and $R_1$ is non-symmetric, we have $R_1=\Delta_{(1,a)}$ with $a>1$, which implies that $R_2=\Delta_{(a,1)}$.
Since $d=2$ and $2=1^*$, by setting $(i,j)=(1,1)$ and $(i,j)=(1,1^{*})$ in Lemma \ref{jb} \ref{jb-1} respectively, we have $k_{1}k_{1}=p_{1,1}^{1}k_{1}+p_{1,1}^{1^{*}}k_{1^{*}}$ and $k_{1}k_{1^{*}}=k_1+p_{1,1^{*}}^{1}k_{1}+p_{1,1^{*}}^{1^{*}}k_{1^{*}}.$ The fact $k_1=k_{1^*}$ implies $k_1=p_{1,1}^{1}+p_{1,1}^{1^{*}}=1+p_{1,1^{*}}^{1}+p_{1,1^{*}}^{1^*}$. In view of Lemma \ref{jb} \ref{jb-2}, one gets $p_{1,1}^{1}=p_{1,1^{*}}^{1^{*}}$, and so $p^{1^{*}}_{1,1}=1+p^{1}_{1,1^{*}}\geq1$. This implies $R_2=R_{1^*}\in R_1^2$, and so $a=2$.  Thus, $\Delta$ is a weakly distance-regular digraph of diameter $2$ with the attached scheme $$(X,\{\Delta_{(0,0)},\Delta_{(1,2)},\Delta_{(2,1)}\}).$$

The desired result follows from Fact \ref{fac2}.
\end{proof}

\begin{lemma}\label{lem3.4}
Suppose that $\mathfrak{X}$ is the attached scheme of $\Gamma$. Then $d=2$, and $\Gamma=(X,R_i)$ for some $i\in\{1,2\}$.
\end{lemma}
\begin{proof}
Since the attached scheme of $\Gamma$ has $2$ classes, we have $d=2$. Then it is immediately from Fact \ref{sub} and Lemma \ref{lem3.3}.
\end{proof}

\subsection{The attached scheme of $\Gamma$ has $3$ classes} In this subsection, we consider the case that the attached scheme of $\Gamma$ has $3$ classes.
\begin{lemma}\label{lem3.5}
If $d=3$, then $p_{1,1}^{1^{*}}\neq 0$ or $p_{1,1}^{3}\neq 0$.
\end{lemma}
\begin{proof}
Suppose, to the contrary that $p_{1,1}^{1^{*}}=0$ and $p_{1,1}^{3}=0$. Then $R_1^2=\{R_1\}$. By setting $(i,j)=(1,1)$ in Lemma \ref{jb} \ref{jb-1}, we obtain $k_1^2=p_{1,1}^1k_1$, and so $p_{1,1}^1=k_1$. Then $p_{1,1}^1+p_{1,0}^1=k_{1}+1$, contrary to Lemma \ref{jb} \ref{jb-3}.
\end{proof}

\begin{lemma}\label{lem3.6}
If $d=3$, then $(X,R_i\cup R_3)$ is a weakly distance-regular digraph of diameter $2$ with $\mathfrak{X}$ as the attached scheme for $i\in\{1,2\}$.
\end{lemma}
\begin{proof}
According to Fact \ref{fac2}, it suffices to prove that $\Delta=(X,R_1\cup R_3)$ is a weakly distance-regular digraph of diameter $2$.

Since $R_1$ is non-symmetric and $R_3$ is symmetric, we have $R_1=\Delta_{(1,a)}$  with $a>1$ and $R_3=\Delta_{(1,1)}$.

Suppose $a>2$. Since $A\Delta=R_1\cup R_3$, one has $p_{1,1}^{1^{*}}=p_{1,3}^{1^{*}}=0$. By Lemma \ref{jb} \ref{jb-2}, we get $p_{1,1}^{3}=p_{1,3}^{1^{*}}=0$, contrary to Lemma \ref{lem3.5}.

Note that $a=2$, and so $R_1=\Delta_{(1,2)}$. It follows that $\Delta$ is a weakly distance-regular digraph of diameter $2$ with the attached scheme $(X,\{\Delta_{(0,0)},\Delta_{(1,2)},\Delta_{(2,1)},\Delta_{(1,1)}\}).$ Thus, the desired result follows.
\end{proof}

\begin{lemma}\label{lem3.7}
Suppose $d=3$. If $\mathfrak{X}$ is neither a $P$-polynomial association scheme nor a wedge product of at least two $2$-class association schemes with the same parameters and a $1$-class association scheme, then $(X,R_i)$ is a weakly distance-regular digraph of diameter $2$ with $\mathfrak{X}$ as the attached scheme for $i\in\{1,2\}$.
\end{lemma}
\begin{proof}
Let $\Delta=(X,R_1)$. Since $R_1$ is non-symmetric and $R_3$ is symmetric, we have $R_1=\Delta_{(1,a)}$  with $a>1$ and $R_3=\Delta_{(b,b)}$  with $b>1$. If $a=b=2$, then $\Delta$ is a weakly distance-regular digraph of diameter $2$ with the attached scheme $(X,\{\Delta_{(0,0)},\Delta_{(1,2)},\Delta_{(2,1)},\Delta_{(2,2)}\})$, which implies that the desired result follows from Fact \ref{fac2}. Therefore, it suffices to show that $a=b=2$.

Suppose, to the contrary that $a>2$ or $b>2$. Since $p_{1,1}^{1^{*}}\neq 0$ or $p_{1,1}^{3}\neq 0$ from Lemma \ref{lem3.5}, we have $a>2$ and $b=2$, or $a=2$ and $b>2$. Note that the induced subdigraph of $\Delta$ on the vertex set $\langle R_1\rangle(x)$ is strongly connected for all $x\in X$. If $a>2$ and $b=2$, then $a=3$ and $(X,\{\Delta_i\}_{i=0}^3)$ is a non-symmetric association scheme since $\tilde{\partial}(\Delta)= \{(0,0),(1,a),(a,1),(2,2)\}$, which imply that $\Delta$ is a distance-regular digraph of girth $4$, contrary to the fact that $\mathfrak{X}$ is not $P$-polynomial.

Note that $a=2$ and $b>2$. The fact $R_3=\Delta_{(b,b)}$ implies $p_{1,1}^{3}=0$. Since $d=3$ and $2=1^*$, by setting $(i,j)=(1,1)$ in Lemma \ref{jb} \ref{jb-1}, we have $k^2_{1}=p_{1,1}^{1}k_{1}+p_{1,1}^{1^{*}}k_{1^{*}}$, and so $k_1=p_{1,1}^1+p_{1,1}^{1^*}$.

In view of Lemma \ref{jb} \ref{jb-2}, we obtain $p_{1,1}^{1}=p_{1,1^{*}}^{1}=p_{1,1^{*}}^{1^{*}}.$ By setting $(i,j)=(1,1^{*})$ in Lemma \ref{jb} \ref{jb-1}, we get $k_1^2=k_{1}k_{1^{*}}=k_1+2p_{1,1}^{1}k_{1}+p_{1,1^{*}}^{3}k_{3}$. By substituting $k_1=p_{1,1}^1+p_{1,1}^{1^*}$ and $p_{1,1^*}^3k_3=p_{3,1}^1k_1$ from Lemma \ref{jb} \ref{jb-2} into the above equation respectively, one obtains
\begin{align}
(p_{1,1}^{1^*})^{2}-p_{1,1}^1-p_{1,1}^{1^*}-(p_{1,1}^1)^{2}-p_{1,1^*}^3 k_3=0\label{bkb}
\end{align}
and $p_{1,3}^1=k_1-2p_{1,1}^1-1$.

By setting $(h,m,e,l)=(1,1,1,1^{*})$ in Lemma \ref{jb} \ref{jb-4}, we have
$$k_{1}+p_{1,1^{*}}^{1^{*}}p_{1,1^{*}}^{1}+p_{1,1^{*}}^{3}p_{1,3}^{1}=p_{1,1}^{1^{*}}p_{1^{*},1^{*}}^{1}.$$
Since $p_{1,1}^{1}=p_{1,1^{*}}^{1}=p_{1,1^{*}}^{1^{*}}$ and $p_{1,1}^{1^{*}}=p_{1^{*},1^{*}}^{1}$ from Lemma \ref{jb} \ref{jb-2}, we obtain
\begin{align}
k_{1}+(p_{1,1}^{1})^2+p_{1,1^{*}}^{3}p_{1,3}^{1}=(p_{1,1}^{1^{*}})^2.\nonumber
\end{align}
Substituting $k_1=p_{1,1}^1+p_{1,1}^{1^*}$ and $p_{1,3}^1=k_1-2p_{1,1}^1-1$ into the above equation, one has
\begin{align}
(p_{1,1}^{1}-p_{1,1}^{1^{*}}+1)(p_{1,1}^{1}-p_{1,1^{*}}^{3}+p_{1,1}^{1^{*}})=0.\nonumber
\end{align}

\textbf{Case 1.} $p_{1,1^{*}}^{3}=p_{1,1}^{1}+p_{1,1}^{1^{*}}$.

Since $k_1=p_{1,1}^{1}+ p_{1,1}^{1^*}$, we have $p_{1,1^{*}}^{3}=k_1$. Pick $(x,y),(y,z)\in R_3$. Suppose $(x,z)\notin R_0\cup R_3$. Then $(x,z)\in R_1\cup R_{1^*}$. Since the proofs are similar, we may assume $(x,z)\in R_1$. Since $p_{1,1^{*}}^{3}=k_1$, we have $z\in R_{1}(x)=P_{1,1^*}(x,y)$, contrary to the fact that $(y,z)\in R_3$. Then $(x,z)\in R_0\cup R_3$. Since $(x,y),(y,z)\in R_3$ were arbitrary, one has $R_3^2\subseteq\{R_0,R_3\}$. By induction, we get $R_3^h\subseteq\{R_0,R_3\}$ for $h>0$, which implies $\langle R_3\rangle =\{R_0,R_3\}$.

Let $\Lambda=\Delta/\langle R_3\rangle$. In view of Lemma \ref{lem} \ref{lem-1}, $\Delta$ is the lexicographic product of $\Lambda$ and the empty graph $\overline{K}_{k_3+1}$. Since $R_3=\Delta_{(b,b)}$, we get $\tilde{\partial}(\Lambda)=\tilde{\partial}(\Delta)\backslash\{(b,b)\}=\{(0,0),(1,2),(2,1)\}$. By Theorem \ref{drdg} and Lemma \ref{lem}, $\Lambda$ is a short distance-regular digraph of girth $3$. In view of the definition of the lexicographic product, one gets $b=3$. Since $(X,\{\Delta_i\}_{i=0}^3\})$ is a non-symmetric association scheme, from Theorem \ref{drdg} again, $\Delta$ is a long distance-regular digraph of girth $3$, contrary to the fact that $\mathfrak{X}$ is not $P$-polynomial.

\textbf{Case 2.} $p_{1,1}^{1^{*}}=p_{1,1}^{1}+1$.

Substituting $p_{1,1}^{1^{*}}=p_{1,1}^{1}+1$ into \eqref{bkb}, one obtains $-p_{1,1^{*}}^{3} k_3=0$. It follows that $p_{1,1^{*}}^{3}=0$, and so $R_1R_{1^*}\subseteq\{R_0,R_1,R_{1^*}\}$. Since $p_{1,1}^3=0$, we obtain $R_1^2\subseteq\{R_1,R_{1^*}\}$. By induction, we get $R_1^h\subseteq\{R_0,R_1,R_{1^*}\}$ for $h>0$, and so $\langle R_1\rangle=\{R_0,R_1,R_{1^*}\}$. Then $p_{l,1}^{3}=p_{l,1^{*}}^{3}=0$ for $l\in \{0,1,2\}$. By Lemma \ref{jb} \ref{jb-3}, we have $k_l=p_{l,3}^{3}$ for $l\in \{0,1,2\}$, which implies $A_lA_3=k_lA_3$ from Lemma \ref{jb} \ref{jb-1}. Thus, $\sum_{l=0}^{2}A_lA_3=\sum_{l=0}^{2}k_lA_3$.

In view of Theorem \ref{th2.3}, $(\langle R_1\rangle(x),\{R_i\cap([\langle R_1\rangle(x)]\times [\langle R_1\rangle(x)])\}_{i=0}^2)$ is a subscheme of $\mathfrak{X}$ for $x\in X$. By the definition of the quotient scheme, $(X/\langle R_1\rangle, \{R_i\}_{i=0}^3 \sslash \langle R_1\rangle)$ is a $1$-class association scheme. By the commutativity of $\mathfrak{X}$, $\mathfrak{X}$ is a wedge product of subschemes $(\langle R_1\rangle(x),\{R_i\cap([\langle R_1\rangle(x)]\times [\langle R_1\rangle(x)])\}_{i=0}^2)$ for each $x\in X$ and a $1$-class association scheme, a contradiction.
\end{proof}

\begin{lemma}\label{lem3.8}
Suppose that $\mathfrak{X}$ is the attached scheme of $\Gamma$. Then $d=3$, and exactly one of the following holds:
\begin{enumerate}
\item\label{lem3.7-1} $A\Gamma=R_i\cup R_3$;
	
\item\label{lem3.7-2} $A\Gamma=R_i$, and $\mathfrak{X}$ is neither a $P$-polynomial association scheme nor a wedge product of at least two $2$-class association schemes with the same parameters and a $1$-class association scheme.
\end{enumerate}
Here, $i\in\{1,2\}$.
\end{lemma}
\begin{proof}
Since the attached scheme of $\Gamma$ has $3$ classes, we have $d=3$. Then $\tilde{\partial}(\Gamma)= \{(0,0),(1,2),(2,1),(2,2)\}$ or $\tilde{\partial}(\Gamma)= \{(0,0),(1,2),(2,1),(1,1)\}$. It follows from Fact \ref{sub} that $A\Gamma=R_i$ or $A\Gamma=R_i\cup R_3$ for $i\in\{1,2\}$. If $A\Gamma=R_i\cup R_3$ for some $i\in\{1,2\}$, from Lemma \ref{lem3.6}, then \ref{lem3.7-1} holds.

Suppose $A\Gamma=R_i$ for some $i\in \{1,2\}$. Note that $R_1$ and $R_2$ are the only pair of non-symmetric relations. If $\mathfrak{X}$ is a $P$-polynomial association scheme, from Theorem \ref{drdg}, then $\Gamma$ is a short distance-regular digraph of girth $4$ or a long distance-regular digraph of girth $3$, which implies that $\Gamma$ has diameter $3$, a contradiction. If $\mathfrak{X}$ is a wedge product of at least two $2$-class association schemes with the same parameters and a $1$-class association scheme, then $\langle R_1\rangle\neq \{R_i\}_{i=0}^3$, which implies that $\Gamma$ is not strongly connected, a contradiction. By Lemma \ref{lem3.7}, \ref{lem3.7-2} is valid.

This completes the proof of this lemma.
\end{proof}

\subsection{The attached scheme of $\Gamma$ has $4$ classes} In this subsection, we consider the case that the attached scheme of $\Gamma$ has $4$ classes.
\begin{lemma}\label{lem3.9}
Suppose $d=4$ and $j\in\{3,4\}$. If $\mathfrak{X}$ is neither a wreath product of a $1$-class association scheme and a $P$-polynomial association scheme nor a wedge product of subschemes $(\langle R_1,R_j\rangle(x),\{R_h\cap ([\langle R_1,R_j\rangle(x)]\times [\langle R_1,R_j\rangle(x)])\}_{h\in \{0,1,2,j\}})$ for each $x\in X$ and a $1$-class association scheme, then $(X,R_i\cup R_j)$ is a weakly distance-regular digraph of diameter $2$ with $\mathfrak{X}$ as the attached scheme for $i\in\{1,2\}$.
\end{lemma}

\begin{proof}
Let $\Delta=(X,R_1\cup R_j)$. Since the proofs are similar, we may assume that $j=3$. Since $R_1$ is non-symmetric and $R_3$, $R_4$ are symmetric, we have $R_1=\Delta_{(1,a)}$  with $a>1$, $R_3=\Delta_{(1,1)}$ and $R_4=\Delta_{(b,b)}$  with $b>1$. If $a=b=2$, then $\Delta$ is a weakly distance-regular digraph of diameter $2$ with the attached scheme $(X,\{\Delta_{(0,0)},\Delta_{(1,2)},\Delta_{(2,1)},\Delta_{(1,1)},\Delta_{(2,2)}\})$, which implies that the desired result is valid from Fact \ref{fac2}. Suppose, to the contrary that $a>2$ or $b>2$.

\textbf{Case 1.} $a>2$.

Since $A\Delta=R_1\cup R_3$, one has $p_{1,1}^{1^{*}}=p_{1,3}^{1^{*}}=p_{3,3}^{1^{*}}=0$. In view of Lemma \ref{jb} \ref{jb-2}, we have $p_{1,1}^{3}=p_{1,3}^{1^{*}}=0$. By setting $(i,j)=(1,1)$ in Lemma \ref{jb} \ref{jb-1}, we obtain
\begin{align}
k_1^2=p_{1,1}^1k_1+p_{1,1}^4k_4.\label{eq5}
\end{align}
If $b>2$, then $p_{1,1}^{4}=0$, and so $k_1^2=p_{1,1}^1k_1$. This implies that $p_{1,1}^1=k_1$, and so $p_{1,1}^1+p_{1,0}^1=k_{1}+1$, contrary to Lemma \ref{jb} \ref{jb-3}. Therefore, $b=2$. Note that the induced subdigraph of $\Delta$ on the vertex set $\langle R_1,R_3 \rangle(x)$ is strongly connected for all $x\in X$. Since $\tilde{\partial}(\Delta)=\{(0,0),(1,a),(a,1),(1,1),(2,2)\}$, one gets $a=3$. Then $\Delta$ is a weakly distance-regular digraph with the attached scheme $$(X,\{\Delta_{(0,0)},\Delta_{(1,3)},\Delta_{(3,1)},\Delta_{(1,1)},\Delta_{(2,2)}\}).$$

By setting $(i,j)=(1,1^*)$ in Lemma \ref{jb} \ref{jb-1}, we obtain $k_1^2=k_{1}k_{1^{*}}=k_1+p_{1,1^{*}}^{1}k_1+p_{1,1^{*}}^{1^*}k_{1^*}+p_{1,1^{*}}^{3}k_3+p_{1,1^{*}}^{4}k_4$. Lemma \ref{jb} \ref{jb-2} implies $p^{1}_{1,1}=p^{1^{*}}_{1,1^{*}}$. By \eqref{eq5}, we have
\begin{align}
	p_{1,1}^{4}k_4=k_1^2-p_{1,1}^1k_1=k_1+p_{1,1^{*}}^{1}k_1+p_{1,1^{*}}^{3}k_3+p_{1,1^{*}}^{4}k_4,\label{eq4}
\end{align}
which implies $p_{1,1}^{4}\neq0.$

Since $p_{1,1}^{1^*}=p_{1,1}^3=0$, by setting $(h,m,e,l)=(3,1,1,1^{*})$ in Lemma \ref{jb} \ref{jb-4}, we get
\begin{align}
p_{1,1^{*}}^{1^{*}}p_{1,1^{*}}^{3}+p_{1,1^{*}}^{3}p_{1,3}^{3}+p_{1,1^{*}}^{4}p_{1,4}^{3}=p_{1,1}^{1}p_{1,1^{*}}^{3}+p_{1,1}^{4}p_{4,1^{*}}^{3}.\label{eq6}
\end{align}
In view of Lemma \ref{jb} \ref{jb-2}, one has $ p_{1,3}^{3}=p_{3,3}^{1^{*}}=0$. By substituting $ p_{1,3}^{3}=0$, $p_{1,4}^{3}=p_{4,1^{*}}^{3}$ and $p_{1,1}^{1}=p_{1,1^{*}}^{1^{*}}$ from Lemma \ref{jb} \ref{jb-2} into \eqref{eq6}, one obtains $p_{1,1^{*}}^{4}p_{1,4}^{3}=p_{1,1}^{4}p_{1,4}^{3}.$
If $p_{1,4}^{3}\neq0$, then $p_{1,1^{*}}^{4}=p_{1,1}^{4},$ which implies $k_1+p_{1,1^{*}}^{1}k_1+p_{1,1^{*}}^{3}k_3=0$ from \eqref{eq4}, a contradiction. Thus $p_{1,4}^{3}=0.$

Since $p_{1,1}^{1^{*}}=p_{1,1}^{3}=0$, by setting $(h,m,e,l)=(3,1,1,3)$ in Lemma \ref{jb} \ref{jb-4}, we get $$p_{1,3}^{1^*}p_{1,1^*}^{3}+p_{1,3}^{3}p_{1,3}^{3}+p_{1,3}^{4}p_{1,4}^{3}=p_{1,1}^{1}p_{1,3}^{3}+p_{1,1}^{4}p_{4,3}^{3}.$$
By substituting $p_{1,3}^{1^{*}}=p_{1,3}^{3}=p_{1,4}^{3}=0$ into the above equation, we obtain $p_{1,1}^{4}p_{4,3}^{3}=0$. Since $p_{1,1}^{4}\neq0$, one obtains $p_{4,3}^{3}=0$. In view of Lemma \ref{jb} \ref{jb-2}, we have $p_{3,3}^{4}=p_{4,3}^{3}=0$. Since $p_{3,3}^{1^{*}}=p_{3,3}^{1}=0$, one gets $R_3^2\subseteq\{R_0,R_3\}$. By induction, we get $R_3^h\subseteq\{R_0,R_3\}$ for $h>0$, which implies $\langle R_3\rangle =\{R_0,R_3\}$. Since $p_{1,1}^3=p_{1,3}^3=p_{1,4}^3=0$, we have $p_{1,1^{*}}^{3}=k_1$ by Lemma \ref{jb} \ref{jb-3}.

Let $\Lambda=\Delta/\langle R_3\rangle$. Note that $\Delta$ is a weakly distance-regular digraph of $\tilde{\partial}(\Delta)=\{(0,0),(1,3),(3,1),(1,1),(2,2)\}$. By Lemma \ref{lem} \ref{lem-2}, $\Delta$ is a lexicographic product of $\Lambda$ and a complete graph, where $\Lambda$ is a weakly distance-regular digraph. Since $R_3=\Delta_{(1,1)}$, from the definition of the lexicographic product, one gets $\tilde{\partial}(\Lambda)=\{(0,0),(1,3),(3,1),(2,2)\}$. By Theorem \ref{drdg}, $\Lambda$ is a short distance-regular digraph of girth $4$. It follows that $\mathfrak{X}$ is a wreath product of a $1$-class association scheme and a $P$-polynomial association scheme, a contradiction.

\textbf{Case 2.} $a=2$.

Note that $b>2$. The fact $R_4=\Delta_{(b,b)}$ implies $p_{1,1}^{4}=p_{1,3}^{4}=p_{3,3}^{4}=0$. By setting $(h,m,e,l)=(4,1,3,3)$ in Lemma \ref{jb} \ref{jb-4}, we have $$p_{3,3}^{1}p_{1,1}^{4}+p_{3,3}^{1^{*}}p_{1,1^{*}}^{4}+p_{3,3}^{3}p_{1,3}^{4}+p_{3,3}^{4}p_{1,4}^{4}=p_{1,3}^{1}p_{1,3}^{4}+p_{1,3}^{1^{*}}p_{1^{*},3}^{4}+p_{1,3}^{3}p_{3,3}^{4}+p_{1,3}^{4}p_{4,3}^{4},$$
and so $p_{3,3}^{1^{*}}p_{1,1^{*}}^{4}=p_{1,3}^{1^{*}}p_{1^{*},3}^{4}.$ In view of Lemma \ref{jb} \ref{jb-2}, one has $p_{1^{*},3}^{4}=p_{1,3}^{4}=0$, which implies $p_{3,3}^{1^{*}}p_{1,1^{*}}^{4}=0.$

\textbf{Case 2.1.} $p_{1,1^{*}}^{4}=0$.

Since $p_{1,1}^4=p_{1,3}^{4}=p_{3,3}^{4}=0$, we obtain $$\{R_1,R_3\}^2= R_1^2\cup R_1R_3 \cup R_3^2\subseteq\{R_0,R_1,R_1^*,R_3\}.$$
Since $p_{1,1^{*}}^{4}=0$, by induction, we get ${(\{R_1,R_3\})}^h\subseteq\{R_0,R_1,R_{1^*},R_3\}$ for $h>0$, which implies $\langle R_1,R_3\rangle=\{R_0,R_1,R_{1^*},R_3\}.$ Then $p_{l,1}^{4}=p_{l,1^{*}}^{4}=p_{l,3}^{4}=0$ for $l\in \{0,1,2,3\}$. By Lemma \ref{jb} \ref{jb-3}, we have $k_l=p_{l,4}^{4}$ for $l\in \{0,1,2,3\}$, which implies $A_lA_4=k_lA_4$ from Lemma \ref{jb} \ref{jb-1}. Thus, $\sum_{l=0}^{3}A_lA_4=\sum_{l=0}^{3}k_lA_4$.

By Theorem \ref{th2.3}, $(\langle R_1,R_3\rangle(x),\{R_h\cap([\langle R_1,R_3\rangle(x)]\times [\langle R_1,R_3\rangle(x)])\}_{h=0}^3)$ is a subscheme of $\mathfrak{X}$ for every $x\in X$. In view of the definition of the quotient scheme, $(X/\langle R_1,R_3\rangle, \{R_i\}_{i=0}^4 \sslash \langle R_1,R_3\rangle)$ is a $1$-class association scheme. By the commutativity of $\mathfrak{X}$, $\mathfrak{X}$ is a wedge product of subschemes $(\langle R_1,R_3\rangle(x),\{R_h\cap([\langle R_1,R_3\rangle(x)]\times [\langle R_1,R_3\rangle(x)])\}_{h=0}^3)$ for each $x\in X$ and a $1$-class association scheme, a contradiction.

\textbf{Case 2.2.} $p_{1,1^{*}}^{4}\neq 0$.

Let $(x,y)\in R_4$. Since $p_{1,1^{*}}^{4}\neq 0$, there exists $z\in P_{1,1^*}(x,y)$. The fact $R_1=\Delta_{(1,2)}$ implies $b=\partial_{\Delta}(x,y)\leq \partial_{\Delta}(x,z)+\partial_{\Delta}(z,y)=3$, and so $R_4=\Delta_{(3,3)}$. It follows that $\Delta$ is a weakly distance-regular digraph with $\tilde{\partial}(\Delta)=\{(0,0),(1,2),(2,1),(1,1),(3,3)\}$.

Since $p_{3,3}^{1^{*}}p_{1,1^{*}}^{4}=0$, one has $p_{3,3}^{1^{*}}=0$. In view of Lemma \ref{jb} \ref{jb-2}, we obtain $p_{3,3}^{1}=0$. Since $p_{3,3}^{4}=0$, one gets $R_3^2\subseteq\{R_0,R_3\}$. By induction, we get $R_3^h\subseteq\{R_0,R_3\}$ for $h>0$, which implies $\langle R_3\rangle =\{R_0,R_3\}$. By setting $(h,m,e,l)=(4,1,1,3)$ in Lemma \ref{jb} \ref{jb-4}, we have $$p_{1,3}^{1}p_{1,1}^{4}+p_{1,3}^{1^{*}}p_{1,1^{*}}^{4}+p_{1,3}^{3}p_{1,3}^{4}+p_{1,3}^{4}p_{1,4}^{4}=p_{1,1}^{1}p_{1,3}^{4}+p_{1,1}^{1^{*}}p_{1^{*},3}^{4}+p_{1,1}^{3}p_{3,3}^{4}+p_{1,1}^{4}p_{4,3}^{4}.$$
By substituting $p_{1,1}^{4}=p_{1,3}^{4}=p_{1^{*},3}^{4}=p_{3,3}^{4}=0$ into the above equation, we obtain $p_{1,3}^{1^{*}}p_{1,1^{*}}^{4}=0,$
which implies $p_{1,3}^{1^{*}}=0$. In view of Lemma \ref{jb} \ref{jb-2}, we have $p_{1,1}^{3}=p_{1,3}^{1^{*}}=0$, $p_{1,3}^3=p_{3,3}^{1^{*}}=0$ and $p_{1,4}^3=p_{1,3}^4=0$. By setting $(i,h)=(1,3)$ in Lemma \ref{jb} \ref{jb-3}, one obtains $p_{1,1^{*}}^{3}=k_1$.

Let $\Lambda=\Delta/\langle R_3\rangle$. Note that $\Delta$ is a weakly distance-regular digraph with $\tilde{\partial}(\Delta)=\{(0,0),(1,2),(2,1),(1,1),(3,3)\}$. By Lemma \ref{lem} \ref{lem-2}, $\Delta$ is a lexicographic product of $\Lambda$ and a complete graph, where $\Lambda$ is weakly distance-regular. Since $R_3=\Delta_{(1,1)}$, from the definition of the lexicographic product, we have $\tilde{\partial}(\Lambda)=\{(0,0),(1,2),(2,1),(3,3)\}$. By Theorem \ref{drdg}, $\Lambda$ is a long distance-regular digraph of girth $3$. Thus, $\mathfrak{X}$ is a wreath product of a $1$-class association scheme and a $P$-polynomial association scheme, a contradiction.
\end{proof}

\begin{lemma}\label{lem3.10}
Suppose that $\mathfrak{X}$ is the attached scheme of $\Gamma$.	Then $d=4$,  $A\Gamma=R_i\cup R_j$ for some $i\in\{1,2\}$ and $j\in\{3,4\}$, and $\mathfrak{X}$ is neither
\begin{enumerate}
\item\label{lem3.9-1} a wreath product of a $1$-class association scheme and a $P$-polynomial association scheme, nor
	
\item\label{lem3.9-2} a wedge product of subschemes $$(\langle R_1,R_j\rangle(x),\{R_h\cap ([\langle R_1,R_j\rangle(x)]\times [\langle R_1,R_j\rangle(x)])\}_{h\in \{0,1,2,j\}})$$ for each $x\in X$ and a $1$-class association scheme.
\end{enumerate}
\end{lemma}

\begin{proof}
Since the attached scheme of $\Gamma$ has $4$ classes, we have $d=4$. Then $\tilde{\partial}(\Gamma)= \{(0,0),(1,2),(2,1),(1,1),(2,2)\}$. It follows from Fact \ref{sub} that $\Gamma=(X,R_i\cup R_{j})$ for some $i\in\{1,2\}$ and $j\in\{3,4\}$.

If $\mathfrak{X}$ is a wreath product of a $1$-class association scheme and a $P$-polynomial association scheme, by Theorem \ref{drdg}, then $\Gamma$ is a lexicographic product of a short distance-regular digraph of girth $4$ or a long distance-regular digraph of girth $3$, and a complete graph, which implies that $\Gamma$ has diameter more than $2$, a contradiction. If $\mathfrak{X}$ is a wedge product of subschemes $(\langle R_1,R_j\rangle(x),\{R_h\cap ([\langle R_1,R_j\rangle(x)]\times [\langle R_1,R_j\rangle(x)])\}_{h\in \{0,1,2,j\}})$ for each $x\in X$ and a $1$-class association scheme, then $\langle R_1,R_j \rangle \neq \{R_h\}_{h=0}^4$, which implies that $\Gamma$ is not strongly connected, a contradiction.

The desired result follows from Lemma \ref{lem3.9}.
\end{proof}

Now, we are ready to prove Theorem \ref{th1}.

\begin{proof}[Proof of Theorem \ref{th1}]
We first prove the sufficiency. By Lemmas \ref{lem3.3}, \ref{lem3.6}, \ref{lem3.7} and \ref{lem3.9}, $(X,A)$ is a weakly distance-regular digraph of diameter $2$ with $\mathfrak{X}$ as its attached scheme, as desired.

We next prove the necessity. In view of \eqref{two way distance set}, the attached scheme of $(X,A)$ has exactly one pair of non-symmetric relations, and has at most $4$ classes. Without loss of generality, we may assume $R_1^{\top}=R_2$. Therefore, by Lemmas \ref{lem3.4}, \ref{lem3.8} and \ref{lem3.10}, the desired result follows.
\end{proof}

\section*{Data Availability Statement}

No data was used for the research described in the article.

\section*{Conflict of Interest}

The authors declare that they have no conflict of interest.

\section*{Acknowledgements}
				
Y.~Yang is supported by NSFC (12101575, 52377162).

\end{document}